\documentclass[11pt,final]{amsart}
\usepackage{stmaryrd,amsfonts,amssymb,amsthm}
\title[An inverse Signorini obstacle problem]{An inverse Signorini obstacle problem}
\author[M. V. de Hoop, M. Lassas, J. Lu, L. Oksanen, Z. Zhao]{Maarten V. de Hoop, Matti Lassas, Jinpeng Lu, Lauri Oksanen, Ziyao Zhao}

\address{Maarten V. de Hoop: Computational and Applied Mathematics and Earth Science, Rice University, Houston, TX 77005, USA}
\email{mdehoop@rice.edu}

\address{Matti Lassas: Department of Mathematics and Statistics, University of Helsinki, FI-00014 Helsinki, Finland}
\email{matti.lassas@helsinki.fi}

\address{Jinpeng Lu: Department of Mathematics and Statistics, University of Helsinki, FI-00014 Helsinki, Finland} 
\email{jinpeng.lu@helsinki.fi}

\address{Lauri Oksanen: Department of Mathematics and Statistics, University of Helsinki, FI-00014 Helsinki, Finland} 
\email{lauri.oksanen@helsinki.fi}

\address{Ziyao Zhao: Department of Mathematics and Statistics, University of Helsinki, FI-00014 Helsinki, Finland} 
\email{ziyao.zhao@helsinki.fi}

\IfFileExists{tweakslo.sty}{\usepackage{tweakslo}}{
\usepackage{amssymb,thmtools,mathtools,todonotes,amsmath,graphicx}
\declaretheorem{theorem,definition,lemma,proposition,corollary,remark}}
\usepackage{xcolor,mdframed}
\usepackage{enumitem}

\def\p{\partial}
\def\o2{\overline{O_2}}
\def\b{\setminus}
\def\R{\mathbb R}

\def\V{\mathcal V}

\def\H{\mathcal H}

\def\tr{\text{tr}}
\def\div{\text{div }}

\def\bs{\boldsymbol{\sigma}}
\def\be{\boldsymbol{\varepsilon}}
\def\bn{\boldsymbol{\nu}}
\def\n{\boldsymbol{n}}
\def\bu{\boldsymbol{u}}
\def\bv{\boldsymbol{v}}
\def\x{\boldsymbol{x}}

\DeclarePairedDelimiter{\norm}{\lVert}{\rVert}
\DeclarePairedDelimiter{\abs}{\lvert}{\rvert}

\newcommand\subsetsim{\mathrel{%
  \ooalign{\raise0.2ex\hbox{$\subset$}\cr\hidewidth\raise-0.8ex\hbox{\scalebox{0.9}{$\sim$}}\hidewidth\cr}}}
  
\newtheorem{main1}{Theorem}
\newtheorem{main2}[main1]{Theorem}

\begin{document}

\begin{abstract}
    We study the inverse problem of determining a Signorini obstacle from boundary measurements for the isotropic elasticity system. 
    We prove that the obstacle can be uniquely determined by a single measurement of displacement and normal stress for the Signorini problem on an open subset of the boundary up to a natural obstruction.
    In addition to considering the Signorini problem, we develop techniques that can be used to study inverse problems for general differential inequalities.
\end{abstract}

\maketitle

\section{Introduction}

The Signorini problem is a classical free boundary problem originally introduced by A. Signorini in linear elasticity \cite{sig1}.
The problem is characterized by the boundary conditions that the solution of the problem at each boundary point must satisfy one of two possible boundary conditions, without a prior knowledge of which condition applies to each point.
Physically, the Signorini conditions model frictionless contact between an elastic body and a rigid support, and they are formulated in terms of inequalities and a nonlinear complementarity condition.
While the original Signorini problem was formulated for the elasticity system, its scalar version has been widely studied as a classical variational problem arising from applications \cite{DL,kikuchi1988contact}.
In this paper, we study the inverse problem of determining a Signorini obstacle from a single boundary measurement for the isotropic elasticity system.
We start by formulating the scalar version of our problem, and then turn to the problem in linear elasticity.

\subsection{Scalar version}
Let $\Omega \subset \R^n$ be a bounded connected open set with smooth boundary, and $O\subset\subset \Omega$ be a connected open subset with smooth boundary modelling an obstacle.
The scalar Signorini problem for the Laplace equation on $\Omega\setminus \overline{O}$ is formulated as follows, with smooth boundary data $f$ on the exterior boundary $\partial \Omega$:
\begin{equation}
\label{eq:direct_problem_lap}
\begin{cases}
\hfil \Delta u=0 \;\textrm{ in }  \Omega\setminus \overline{O}, \\
\hfil u|_{\partial \Omega}=f, \\
\hfil u\geq 0, \;\; \partial_{\nu}u\geq 0, \;\; u\partial_{\nu} u =0 \textrm{ on } \partial O,    
\end{cases}
\end{equation}
where the normal derivative $\partial_{\nu} u|_{\partial O}$ is taken with respect to the outward unit normal $\bn$ for $\Omega\setminus \overline{O}$ at $\partial O$.
This is also known as the thin obstacle problem that appears in studying classical obstacle problems \cite{figalli2019regularity,ros2018obstacle,survey-obstacle} when the constraint is only imposed for the boundary (a hypersurface).

In general the Signorini problem \eqref{eq:direct_problem_lap} does not have smooth solutions up to the boundary of the obstacle.
%even if all the problem data are smooth.
The weak formulation of the problem \eqref{eq:direct_problem_lap} is understood as a variational problem of finding $u\in K$ where
$$K:=\big\{v\in H^1(U): v|_{\partial O} \geq 0, \; v|_{\partial \Omega}=f \big\},\quad U:=\Omega\setminus \overline{O},$$
such that 
\begin{equation}
\int_{U} \nabla u \cdot \nabla (v-u) dx \geq 0, \; \, \textrm{ for all }v\in K.
\end{equation}
A general theory for the existence and uniqueness of solution to the variational problem was developed by Lions-Stampacchia in \cite{LS}.
%uniqueness is directly by variational problem. Add variational inequality for two solutions, prove that they differ by constant which has to be zero due to fixed boundary value.
It was known due to Frehse \cite{Frehse1977} and Caffarelli \cite{Caffarelli1979} that the solution to the scalar Signorini problem is in $H^2(U)\cap C^{1,\alpha}(\overline{U})$ for some $\alpha\leq 1/2$.
The optimal $C^{1,\frac12}$-regularity was proved years later by Athanasopoulos-Caﬀarelli \cite{athanasopoulos}.
%\cite[Lemma 2.2]{Frehse1977} for $H^2$ regularity in the scalar case.
%Theorem 2.2 in Kinderlehrer's paper explicitly states global $H^2$ regularity in more complicated elasticity case, except near switching of boundary conditions, which is not present in our setting.

We consider an inverse obstacle problem: given a boundary value $f$, does the normal derivative $\partial_{n} u|_{\Gamma}$ on an open subset $\Gamma\subset \partial \Omega$ of the exterior boundary $\partial \Omega$ uniquely determine the obstacle $O$?
Observe that the inverse problem cannot be uniquely solved when the given boundary value $f$ is a nonnegative constant function, in which case $u$ is the same (nonnegative) constant everywhere so the normal derivative is identically zero no matter what the obstacle is.
We prove that this is the only obstruction in solving the inverse obstacle problem.

\begin{main1}
\label{main-Laplace}
\begin{mdframed}[backgroundcolor=black!4, linewidth=0pt, innerleftmargin=0pt, innerrightmargin=0pt] 
{\it Let $\Omega\subset \mathbb{R}^n$ be a bounded connected open set with smooth boundary and $\Gamma\subset \partial \Omega$ be a nonempty open subset. Let $O_1,O_2\subset\subset \Omega$ be (possibly empty) open subsets with smooth boundary, representing the obstacles.
Assume that $\Omega\setminus \overline{O_1},\Omega\setminus \overline{O_2}$ are connected. 
Suppose $u_1,u_2$ solves the scalar Signorini problem \eqref{eq:direct_problem_lap} for a given non-constant function $f$ with obstacles $O_1,O_2$, respectively.
If $\partial_{n} u_1|_{\Gamma}=\partial_{n} u_2|_{\Gamma}$ with respect to the unit normal $\boldsymbol{n}$ to $\partial \Omega$, then $O_1=O_2$.}
\end{mdframed}
\end{main1}

\begin{remark}
As a special case of Theorem \ref{main-Laplace}, the boundary data determine if there is an (smooth) obstacle or not.
%Note that our method only applies to obstacles with smooth boundary.
%so we do not know if there is non-smooth obstacle inside, say less than $C^2$.
The conclusion of Theorem \ref{main-Laplace} is still valid if $f$ is a negative constant, which is stated as Proposition \ref{prop:nec_lap}.
\end{remark}

%The complement is connected condition is necessary, otherwise impossible to determine the inner shape.

%The necessarity for assumption that $f$ is non-constant is justified in Proposition \ref{prop:nec_lap}.

\subsection{Linear elasticity} 
\label{subsection_intro_elastic}

Consider that an elastic body occupying $\Omega\setminus \overline{O}$, in an equilibrium configuration, is constrained on $\p\Omega$ and rests on a frictionless rigid body $O$. Let us denote by $\bu:\Omega\setminus \overline{O}\to \R^n$ the displacement vector of the elastic body and by $\be(\bu)=\frac{1}{2}(\nabla \bu+(\nabla \bu)^T)$ the linearized strain tensor.
Then the stress tensor for the Lam\'e system is given by 
\begin{equation}
    \bs(\bu)=2\mu \be(\bu)+\lambda\, \tr(\be(\bu))I_n,
\end{equation}
where $\mu (x),\lambda(x)\in C^\infty(\overline{\Omega})$ are positive smooth Lam\'e coefficients.
Let us denote by $\bn$ and $\boldsymbol{n}$ the outward unit normal for $\Omega \setminus \overline{O}$ at $\partial O$ and $\partial \Omega$, respectively. 
At each point of $\p O$, one of the following two conditions holds (see e.g. \cite{DL,sofonea}):
\begin{equation}
    \label{eq:sig_bc}
    \begin{cases}
        \bs(\bu)_\tau=0,\ \bs(\bu)_\nu\leq 0,\ \bu_\nu = 0,\\
        \qquad\qquad\qquad \text{or}\\
        \bs(\bu)_\tau=0,\ \bs(\bu)_\nu=0,\ \bu_\nu < 0,
    \end{cases}
\end{equation}
where $\bu_\nu=\bu\cdot\bn$ is the normal displacement, $\bs(\bu)_\nu=\bs(\bu)\bn\cdot \bn$ is the normal stress, and $\bs(\bu)_\tau=\bs(\bu)\bn-\bs(\bu)_\nu \bn$ is the tangential stress (being zero due to no friction). 
The first set of conditions above is satisfied in the region where the elastic body is constrained by $\p O$, meaning that it is in the contact region and no displacement occurs in the normal direction. The second set of conditions holds in the region where normal displacement occurs, indicating the absence of an active obstacle and, consequently, no normal stress.
Note that in our setting we turn the typical formulation inside out: a deformable body $\Omega \setminus \overline{O}$ surrounding a rigid obstacle $O$.
Then the classical formulation of the model, with prescribed smooth displacement $\boldsymbol{f}$ on the exterior boundary $\partial \Omega$, is
\begin{equation}
    \label{eq:direct_problem_ela}
    \begin{cases}
    \hfil  \div \bs(\bu)=0, \text{ in }\Omega\setminus \overline{O},\\
    \hfil \bu=\boldsymbol{f}, \text{ on }\p\Omega,\\
    \hfil \bs(\bu)_\tau=0,\ \bu_\nu \leq 0,\ \bs(\bu)_\nu \leq 0,\ \bu_\nu\bs(\bu)_\nu=0,\ \text{ on }\p O.
    \end{cases}
  \end{equation}  
The Signorini problem in linear elasticity is to find a solution to the system \eqref{eq:direct_problem_ela}.

Similar to the scalar case, 
the system \eqref{eq:direct_problem_ela} does not have smooth solutions in general, and the problem
is understood as a variational problem of finding $\bu\in \boldsymbol{K}$ where
%the exact weak formulation with prescribed displacement is in Kinderlehrer's paper.
\begin{equation}
    \boldsymbol{K}:=\big\{\bv\in (H^1(U))^n\mid \bv\cdot \nu\leq 0 \text{ on }\p O,\ \bv=\boldsymbol{f}\text{ on }\p\Omega \big\},\quad U:=\Omega\setminus \overline{O},
\end{equation}
such that
\begin{equation}
    \int_{U} \bs(\bu) : (\be(\bv)-\be(\bu)) \,dx\geq 0,\ \text{for all }\bv\in \boldsymbol{K},
\end{equation}
where the operation $:$ is defined as $\bs:\be=\sum_{i,j=1}^n \sigma_{ij}\varepsilon_{ij}$, where $\sigma_{ij}$, $\varepsilon_{ij}$ are the components of the tensors $\bs,\,\be$, respectively. 
The existence and uniqueness of the weak solution to the variational problem above were studied by Fichera \cite{fichera2}, see also \cite{LS,sofonea}, and the solution is in $H^2(U)$ due to Kinderlehrer \cite{kinderlehrerF}.
In \cite{schumann1989} Schumann proved the $C^{1,\alpha}$-regularity of the solution in general dimensions $n\geq 2$ for some $\alpha>0$.
In dimension 3, the optimal $C^{1,\frac12}$-regularity was proved by Andersson \cite{john2016}.

\begin{figure}[h]
  \begin{center}
    \includegraphics[width=0.6\linewidth]{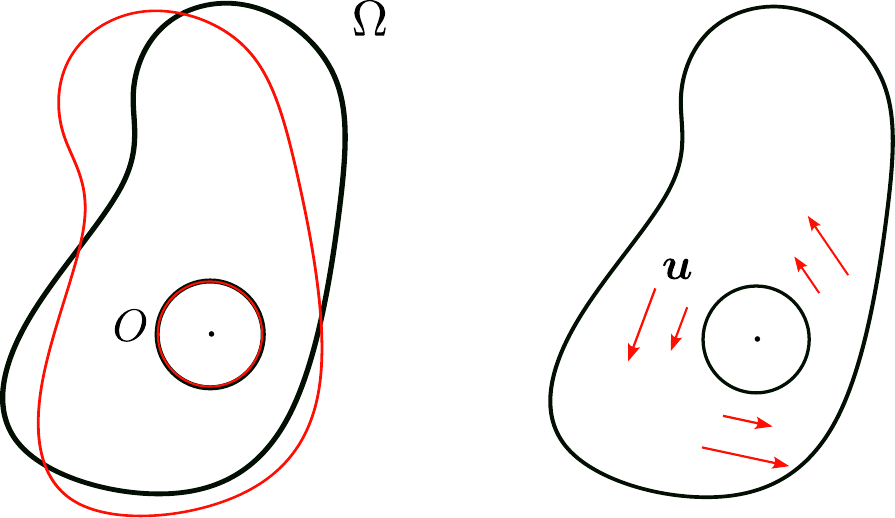}
    \caption{Boundary measurements of the displacement $\boldsymbol{u}$ on $\partial \Omega$ resulting from rotations cannot distinguish the size of a round Signorini obstacle $O$, in which case the stress tensor is identically zero. The right figure illustrates the displacement vector $\boldsymbol{u}$ resulting from a (infinitesimal) rotation with respect to the center of the obstacle. The displacement on $\partial O$ is perpendicular to the normal direction so the Signorini contact conditions on $\partial O$ are valid.}
    \label{fig_rotation}
  \end{center}
\end{figure}

We study the inverse obstacle problem: given boundary data $\boldsymbol{f}$, does the normal stress $\bs(\bu)\n|_{\Gamma}$ on a subset $\Gamma\subset \partial \Omega$ of the exterior boundary $\p \Omega$ uniquely determine the obstacle $O$? 
%The answer to this problem is given by Theorem \ref{main:elasticity}.
In general this is not possible if the given boundary data $\boldsymbol{f}$ is allowed to take the form of a rigid motion defined by
%and $\boldsymbol{f}\not\in \mathcal{R}$, where $\mathcal{R}$ is the rigid motion defined as
\begin{equation}
\label{rigid_motion}
    \mathcal{R}:=\big\{\boldsymbol{c}+A \x\mid \boldsymbol{c}\in \R^n,\ A\in \R^{n\times n},\ A^T=-A,\ \x\in\p\Omega \big\}.
\end{equation}
A counterexample to the unique determination is explained in Figure \ref{fig_rotation}, where the deformation of an elastic body corresponds to a rotation with respect to the center of a round obstacle $O$. The displacement vector is given by $\bu=A\x$, where $A$ is a skew-symmetric matrix and $\x$ is the position vector. In this case, the Signorini boundary conditions are satisfied on any sphere (with the same center) regardless of its radius, so the radius of the obstacle cannot be detected from measurements on the exterior boundary.
%Details of this example are given in Section \ref{sec-counterexamples}.

We are able to fully solve the inverse obstacle problem up to this natural obstruction.

\begin{main2}
\label{main:elasticity}
\begin{mdframed}[backgroundcolor=black!4, linewidth=0pt, innerleftmargin=0pt, innerrightmargin=0pt] 
{\it Let $\Omega\subset \mathbb{R}^n$ be a bounded connected open set with smooth boundary and $\Gamma\subset \partial \Omega$ be a nonempty open subset. Let $O_1,O_2\subset\subset \Omega$ be (possibly empty) open subsets with smooth boundary, representing the obstacles.
Assume that $\Omega\setminus \overline{O_1},\Omega\setminus \overline{O_2}$ are connected.
    Suppose $\bu_1, \bu_2$ solves the Signorini problem \eqref{eq:direct_problem_ela} for $\boldsymbol{f}\not\in \mathcal{R}$ with obstacles $O_1, O_2$, respectively. If $\bs(\bu_1)\n|_{\Gamma}=\bs(\bu_2)\n |_{\Gamma}$, then $O_1=O_2$.}
    \end{mdframed}
\end{main2}

\begin{remark}
%Theorem \ref{main:elasticity} is valid if one of the obstacles is empty.
In the case of $\boldsymbol{f}=\boldsymbol{c}+A\x \in\mathcal{R}$ for some constant vector $\boldsymbol{c}\in \mathbb{R}^n$ and skew-symmetric matrix $A$, the solvability of the inverse problem depends on $\boldsymbol{c}$ and $A$.
For example, when $|\boldsymbol{c}|$ is larger than the range of $A$, the inverse obstacle problem is uniquely solvable. A characterization of unique solvability in this case is stated in Proposition \ref{prop:nec_ela}.
%Lemma \ref{prop:nec_ela}
\end{remark}

%We note that $\boldsymbol{f}\not\in \mathcal{R}$ is necessary by Proposition \ref{prop:nec_ela}.

\smallskip

We prove Theorem \ref{main-Laplace} and \ref{main:elasticity} by arguing that both $O_1\setminus O_2$ and $O_2\setminus O_1$ are empty. We prove this by contradiction, with a unique continuation approach that dates back to Schiffer \cite{lax}.
The main difficulty lies in the fact that the intersection between two open sets may have rough boundary, even if the two open sets both have smooth boundary. First, the boundary of the intersection is not necessarily piecewise smooth and could have infinitely many connected components. 
Indeed, this can happen when one of the open set has wildly oscillating boundary of the form $e^{-1/x^2}\sin(1/x)$.
Moreover, the set of boundary points at which the unit normal is well-defined can be significantly smaller than the whole boundary, for example considering the disk with a slit: $\{(r,\theta): 0<r<1,\, 1<\theta <2\pi\}$ in the polar coordinate of $\mathbb{R}^2$.
The roughness of the boundary causes essential difficulties in utilizing contact conditions on the difference of the obstacles.
To make matters worse, the Signorini contact conditions impose a ceiling on the regularity of the solutions, $C^{1,\alpha}$-regularity to be exact,
even if the boundary data are smooth. 
To handle these difficulties, we operate using the theory of functions of bounded variations and sets of finite perimeter.
In Section \ref{sec:3}, we briefly review these concepts and show how they apply to our setting. 
Then we prove our main results in Section \ref{sec:4} and \ref{sec:5}.

\subsection{Motivations and related results}

Inverse obstacle problems study the detection of obstacles (or inclusions) from boundary measurements generated by physical fields, such as electrical, acoustic or thermal boundary measurements.
There is vast literature on inverse obstacle problems \cite{Isakov09} and we focus on theoretical results on these problems for elliptic equations. 
From single boundary measurement, one of the first uniqueness proofs for identifying obstacles is due to Schiffer \cite{lax} originally used in inverse scattering. 
A more recent work \cite{Bacchelli09} used this method to prove the unique recovery of obstacle with the Robin boundary condition.
In the case of many boundary measurements, Isakov proved the unique determination of obstacle and isotropic conductivity from the Dirichlet-to-Neumann map \cite{Isakov1988}.
Logarithmic stability estimate for the determination of obstacle was obtained in \cite{Alessandrini05}.
Obstacle detection in isotropic medium using complex geometrical optics solutions was studied in \cite{Hiroshi07,Gunther07,Wang09}.
In the anisotropic case, the unique determination of obstacle was studied in \cite{Kwon04} using fundamental solutions for constant medium, and
\cite{Gunther05} using Runge approximation.
Constructive methods to find obstacles include \cite{Bourgeois10} for single measurement and \cite{Ikehata98} for many measurements.

%{\color{red} [Do we need to include this paragraph? as these inverse problems are very different from the one we consider here?]} 

The Signorini contact condition considered in the present work is a classical way to model frictionless contact of an elastic body.
It defines a free boundary problem in the sense that the actual region of contact is not known a priori and needs to be determined as part of the problem. Free boundary problems naturally occur in physics, for instance in fluid dynamics, contact and fracture mechanics \cite{DL,kikuchi1988contact,eck2005unilateral}. When friction is present, in the dynamic setting, the analysis of contact models becomes more complicated. For (visco)elastic bodies, a physically reasonable Coulomb friction law has been defined in a nonlocal way by mollifying the normal force, which guarantees well-posedness; see \cite{tani2020dynamic} for a detailed model including the Signorini contact condition.
%and one main category of interest has been the obstacle problems \cite{figalli2019regularity,ros2018obstacle,survey-obstacle}.

In general, these problems can be mathematically understood through variational inequalities (or minimization of energy functionals) over a set of constraints, giving rise to a system of differential inequalities instead of the classical Euler-Lagrange equations.
Inverse problems have been extensively studied from the perspective of PDEs and differential geometry, but are almost unknown for differential inequalities mainly due to their nonlinear and complex nature. To the best of our knowledge, our present paper serves as a first study into inverse problems for free boundary contact models.
As many physically relevant problems involve contact conditions, the classical inverse problems for PDEs could have generalizations to (partial) differential inequalities. We hope that the present paper initiates new research directions and develops techniques suitable for the future study of inverse problems for (partial) differential inequalities.

More traditional inverse problems for the constitutive parameters of elastic materials have a long history. 
Inverse problems of determining the stiffness tensor in the elasticity system from boundary measurements have been widely studied for both static and dynamic cases.
For the static problem, the determination is only known for isotropic cases when the Lam\'e parameters are close to constants \cite{Nakamura93,Gunther94,Nakamura03,Eskin02,Gunther12}, and the stability was analyzed in \cite{Beretta14}. For the dynamic problem, the unique determination of stiffness tensor is known for isotropic elasticity
\cite{Rachele00,Rachele00_1,Maarten17,Stefanov18}.  The dynamic inverse rupture problem in an isotropic elastic medium, and its stability, were analyzed in \cite{Maarten23}. The determination of anisotropic elasticity was studied in \cite{Mazzucato06,Maarten19,de2020recovery,Lauri20} and recent works \cite{Joonas23,ilmavirta2023gauge}.

\medskip
\noindent {\bf Acknowledgement.} 
The authors thank Laurent Bourgeois and J\'er\'emi Dard\'e for helpful discussions.
MVdH was supported by the Simons Foundation under the MATH + X program, the National Science Foundation under grant DMS-2108175, and the corporate members of the Geo-Mathematical Imaging Group at Rice University.
M.L. and J.L. were supported by the PDE-Inverse project of the European Research Council of the European Union, project 
101097198, and the Research Council of Finland, grants 273979 and 284715. 
L.O. was supported by the European Research Council of the European Union, grant 101086697 (LoCal), and the Research Council of Finland, grants 359182, 347715 and 353096.
Z.Z. was supported by the Finnish Ministry of Education and Culture’s Pilot for Doctoral Programmes (Pilot project Mathematics of Sensing, Imaging and Modelling).
Views and opinions expressed are those of the authors only and do not necessarily reflect those of the European Union or the other funding organizations.

\section{Preliminary constructions}
\label{sec:2}

The basic idea of our method is to consider the elasticity system on the difference of the obstacles $O_1\setminus O_2$ with the Signorini conditions on its boundary, and produce a contradiction with the exterior boundary measurement through unique continuation.
However, as explained above, the set $O_1\setminus O_2$ may be complicated even if $O_1,O_2$ have smooth boundary. Moreover, the complement of $O_1\cup O_2$ may be disconnected, see Figure \ref{fig_domains} (left), which creates an inner region not accessible from the exterior boundary $\partial \Omega$.
This makes it difficult to determine the shape in the inaccessible region from exterior boundary measurements.

\begin{figure}[h]
\label{fig_domains}
\begin{minipage}[h]{0.45\linewidth}
    \centering
    \includegraphics[width=6cm]{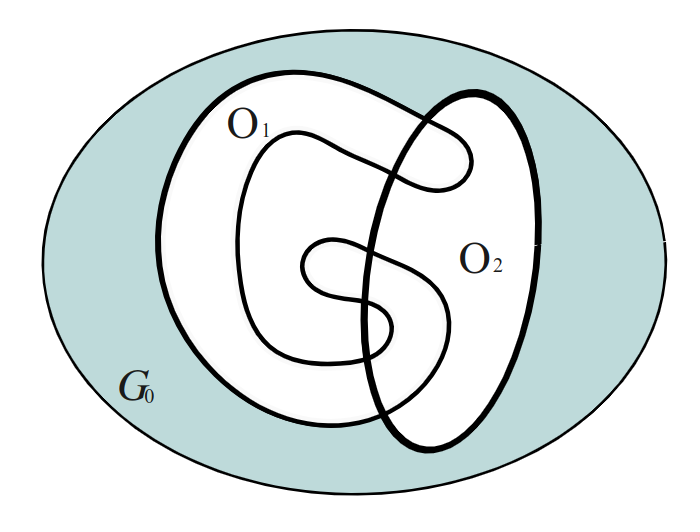}
\end{minipage}
\hfill
\begin{minipage}[h]{0.45\linewidth}
    \includegraphics[width=6cm]{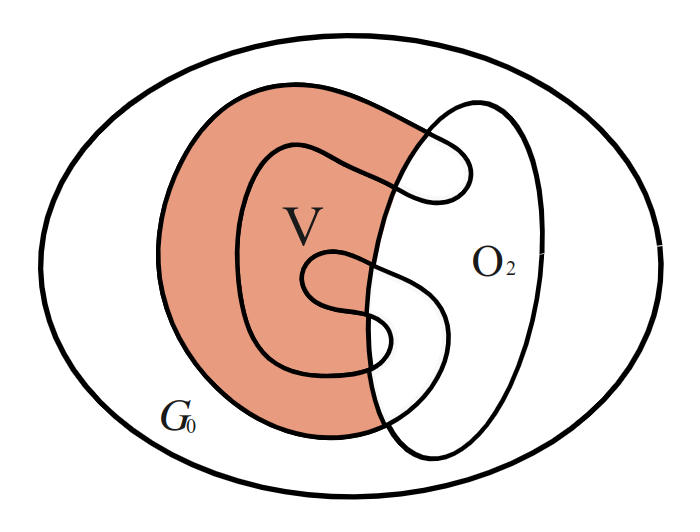}
\end{minipage}
\caption{An illustration of the sets $G_0$ and $\mathcal{V}$.}
\end{figure}

In this section, we construct appropriate domains to carry out our method.
Let $\Omega\subset \R^n$ be a bounded open set with smooth boundary, and let $O_1,O_2$ be two open subsets with smooth boundary satisfying:
\vspace{-1mm}
%One of assumption (2) can be assumed without loss of generality.
\begin{enumerate}
  \item[1.] $O_1\subset\subset \Omega$ and $O_2\subset\subset \Omega$,
  \item[2.] $O_1\not\subset O_2$,
  \item[3.] $\Omega\setminus \overline{O_1}$ and $\Omega\setminus \overline{O_2}$ are connected.
\end{enumerate}
\noindent We define the following sets:
\begin{align} 
  G_0:= \textrm{the connected component of } \Omega\setminus (\overline{O_1\cup O_2}) \textrm{ such that } \p \Omega\subset \p G_0, \label{def-G0} \\
  \V:= \textrm{a connected component of } (\Omega\setminus \overline{G_0})\setminus \overline{O_2} \textrm{ satisfying } \p\V \cap \p G_0\neq \emptyset. \label{def-Vset}
\end{align}
An illustration of these sets is given in Figure \ref{fig_domains}.

Roughly speaking, $G_0$ represents the largest domain in which unique continuation can be propagated from exterior boundary data, and $\V$ is a set containing a connected component of $O_1\setminus \overline{O_2}$ that is connected to the exterior boundary.
%We note that the set $G_0$ is nonempty and unique due to our setting. 
However, it is not clear from the definitions if there exists a connected component satisfying the definition of $\mathcal{V}$. In the following (Lemma \ref{prop:2}) we show that one can always have a nonempty choice of the set $\mathcal{V}$.\par

  \begin{lemma}
    \label{prop:1}
    $\quad \p G_0\cap (\p O_1\b \p O_2)\neq \emptyset$.
  \end{lemma}
  \begin{proof}
  Suppose $\p G_0\cap (\p O_1\b \p O_2)= \emptyset$, then $\p G_0\subset \p \Omega\cup \p O_2=\p(\Omega\b \o2)$. According to Lemma \ref{lm:1} in Appendix \ref{apx}, there holds $G_0=\Omega\b \o2$. However, by definition we have $G_0\subset \Omega\setminus (\overline{O_1\cup O_2})=(\Omega\setminus \o2)\setminus (O_1\setminus\o2)$ and it is sufficient to show that $O_1\setminus \o2\neq \emptyset$. Indeed, suppose $O_1\setminus \o2= \emptyset$, then $O_1\subset O_2\cup \p O_2$ and there exists a point $p\in O_1\cap \p O_2$ as $O_1\not\subset O_2$. Since $\p O_2$ is smooth, any neighborhood $U$ of $p$ intersects $\Omega\setminus \overline{O_2}$ and for sufficiently small neighborhood $U_p$ there holds $U_p\subset O_1$. Thus $\emptyset\neq U_p\cap (\Omega\setminus \overline{O_2})\subset O_1\cap (\Omega\setminus \overline{O_2})$, which contradicts with $O_1\subset \overline{O_2}$.
  \end{proof}\par
  \begin{lemma}
      \label{lm:new}
      For any point $z\in \p G_0\cap (\p O_1\b \p O_2)$, there exists a small neighborhood $U_z$ of $z$ such that $U_z\setminus \overline{O_1}\subset G_0$.
  \end{lemma}
  \begin{proof}
      First we show that $z\not\in \overline{O_2}$. If this is not true, we have $z\in O_2$ as $z\not\in \p O_2$ by setting. However, this contradicts with $z\in \p G_0\subset \overline{G_0}\subset \overline{\Omega}\setminus(O_1\cup O_2)$.\par
      Since $O_1$ is smooth and $z\in \p O_1$, for small enough neighborhood $U_z$ of $z$ there holds $U_z$ is a local coordinate chart near $z$ and $U_z\cap \overline{O_2}=\emptyset$. To simplify the notation, we write $U_z^-:= U_z\setminus \overline{O_1}$. Since $z\in \p G_0$, there exists a sequence of points $z_j\in G_0$ such that $z_j\to z$. Then for sufficiently large $J$ we have $z_J\in U_z^-$, since $G_0\cap \overline{O_1}=\emptyset$ by definition. Thus, for any $p\in U_z^-$, one can connect $p$ with $z_J$ by a path in $U_z^-$. 
        %because $U_z$ is two-side collar coordinate neighborhood.
        Since $U_z^-$ does not intersect with $\overline{O_1}$ or $\overline{O_2}$, 
        %so that the path is contained in $\Omega\setminus (\overline{O_1\cup O_2})$.
        we see that $p\in G_0$, which proves the statement as $p\in U_z^-$ is arbitrary.
  \end{proof}
  Lemma \ref{prop:1} indicates that $\p G_0\setminus \p O_2$ is not empty, which allows us to construct a non-empty $\V$.

  \begin{lemma}
    \label{prop:2}
    There is a nonempty connected component $\V$ of $(\Omega\setminus \overline{G_0})\setminus \overline{O_2}$ satisfying $\p\V \cap \p G_0\neq \emptyset$.
  \end{lemma}
  \begin{proof}
  According to Lemma \ref{prop:1}, there exists a point $p\in \p G_0\cap (\p O_1\b \p O_2)$, and we can find an open ball $B(p,\delta)$ centered at $p$ with radius $\delta$ such that $B(p,\delta)\cap \o2=\emptyset$. If such ball does not exist, that is, for any $\varepsilon>0$, there holds $B(p,\varepsilon)\cap \o2\neq\emptyset$. This means that $p\in \o2$. Notice that $p\not\in \p O_2$, then $p\in O_2$ and there is an neighborhood $U_p$ of $p$ such that $U_p\subset O_2$ and $U_p\cap G_0=\emptyset$, then $p\not\in \p G_0$, a contradiction. \par
  Furthermore, since $O_1$ is a domain with smooth boundary and $p\in \p O_1$, there exists $\eta>0$ such that $B(p,\eta)\cap O_1$ is connected by the local connectedness of $\R^n$. Then we choose $\gamma=\min\{\delta,\eta\}$, thus $\left(B(p,\gamma)\cap (\Omega\b \overline{G_0})\right)\subset (\Omega\b \overline{G_0})\b \o2$, and we can let $\V$ to be a connected component of $(\Omega\b \overline{G_0})\b \o2$ containing $B(p,\gamma)\cap O_1$. It remains to verify that $p\in \p \V$. Since $p\in \p O_1$, then we can find a sequence $\{x_n\}_{n\geq 1}\subset O_1$ that converges to $p$, and there exists a number $N$ such that $\{x_n\}_{n\geq N}\subset B(p,\gamma)$. Thus $\{x_n\}_{n\geq N}\subset B(p,\gamma)\cap O_1\subset \V$, and $p\in \overline{\V}$. To get a contradiction, assume $p\in \V$, then $p\in \Omega\b \overline{G_0}$, and in particular $p \not\in\p G_0$.
  \end{proof}
  
  Before concluding this section, we state the following lemma, which plays an important role in proving the main results. The lemma enables us to connect $\V$ to $G_0$, as well as $\p\Omega$, through a path that does not pass $O_2$. %In other words, the following lemma shows that $\V$ and $G_0$ are still connected in $\Omega\setminus\overline{O_2}$.
    \begin{lemma}
        \label{V-boundary}
        $\quad \emptyset \neq\, \partial \mathcal{V} \setminus \partial O_2 \subset \partial G_0\cap \p O_1.$
    \end{lemma}
\begin{proof}
 Notice that
  \begin{equation}
  \label{eq:lm_3}
    \p\V\subset \p(\Omega \setminus \overline{G_0})\cup \p(\Omega \setminus \o2)\subset \p \Omega\cup \p G_0\cup \p O_2\subset \p\Omega\cup \p O_1 \cup \p O_2.
  \end{equation}
    Since $G_0$ contains a neighborhood of $\p \Omega$ by the local connectedness, we have $\p \V\cap \p\Omega=\emptyset$ and $\p\V\subset\p O_1\cap\p O_2$. Assume $\p \V\setminus \p O_2=\emptyset$, then $\p \V\subset \p O_2\subset \p (\Omega\b \o2)$. By Lemma \ref{lm:1} in Appendix \ref{apx}, $\V=\Omega\b \o2$ as $\Omega\setminus \overline{O_2}$ is connected, which is a contradiction with $G_0\neq \emptyset$.\par
    Moreover, \eqref{eq:lm_3} implies that for any $z\in \partial \mathcal{V} \setminus \partial O_2$ satisfies that $z\in \partial (\Omega\setminus G_0)\subset \partial \Omega\cup \partial G_0$. Thus $\partial \mathcal{V} \setminus \partial O_2 \subset \partial G_0$ follows immediately from the fact $\p\V\cap\p\Omega=\emptyset$.
\end{proof}

    The boundary of $\mathcal{V}$ essentially consists of unions and intersections of manifold boundaries. In general, even when the manifold boundaries are smooth, their unions or intersections can behave wildly, and $\p\V$ is not necessarily piecewise smooth. In fact, $\V$ is a set of finite perimeter and the measure-theoretic unit normal on its reduced boundary is well-defined. We will briefly review these concepts in the next section, and show that the measure-theoretic unit normal on the reduced boundary of $\V$ coincides with that of manifold boundaries almost everywhere. 
%    We explain these concepts in the next section.

\section{Set of finite perimeter and reduced boundary}
\label{sec:3}

In this section, our primary objective is to show that $\V$ is a set of finite perimeter and that its measure-theoretic unit normal coincides with that of $O_1$ or $O_2$, possibly differing by a sign.
For self-containedness, we briefly review basic concepts on sets of finite perimeter in geometric measure theory. 
%The classical approach to these concepts is in the context of functions of bounded variation.

Let $u\in L^1(\Omega)$, we say that $u$ is a function of bounded variation in $\Omega$, written as $u\in BV(\Omega)$, if its distributional gradient $Du$ is a finite $\R^n$ vector-valued Radon measure $Du=(\mu_1,\mu_2,\cdots,\mu_n)$ such that
\begin{equation*}
    \int_\Omega u\frac{\p \varphi}{\p x_i}dx=-\int_\Omega \varphi d\mu_i,\quad \forall \varphi\in C_0^\infty(\Omega),\ i=1,\cdots,n,
\end{equation*}
and
\begin{equation*}
    \abs{Du}(\Omega):=\sup\left\{\int_\Omega u\, \div \varphi\ dx\ \big|\ \varphi\in\left(C_0^\infty(\Omega) \right)^n,\ \norm{\varphi}_{L^\infty(\Omega)}\leq 1 \right\}<\infty.
\end{equation*}
Let $E$ be an measurable subset of $\R^n$. For any open set $\Omega\subset \R^n$, $E$ is said to be a set of finite perimeter in $\Omega$ if the characteristic function $\chi_E\in BV(\Omega)$. When $E$ is of finite perimeter in $\R^n$, it is simply called a set of finite perimeter. \par
For a set $E$ of finite perimeter, the \emph{reduced boundary} of $E$, denoted by $\p^* E$, is defined as the set of all points $x\in \R^n$ such that $\abs{D\chi_E}(B(x,r))>0$ for all $r>0$ and the limit
\begin{equation*}
    \nu_E(x):=\lim_{r\to 0}\frac{D\chi_E(B(x,r))}{\abs{D\chi_E}(B(x,r))}
\end{equation*}
exists in $\R^n$ with $\abs{\nu_E(x)}=1$. 
The vector $\nu_E(x)$ is called the \emph{measure-theoretic unit normal} to $E$ at $x$.\par
The set of points of density $t$ of $E$ is defined as
\begin{equation}
    E^{(t)}=\left\{x\in \R^n\mid \lim_{r\to0^+}\frac{\abs{E\cap B(x,r)}}{\abs{B(x,r)}}=t \right\},
\end{equation}
It is clear that the interiors of $E$ and $\R^n\setminus E$ satisfy $E^\circ\subset E^{(1)}$ and $(\R^n\setminus E)^\circ\subset E^{(0)}$, respectively. Moreover, Federer's theorem (see e.g. \cite[Theorem 16.21]{maggi} or \cite[theorem 3.61]{ambrosio}) states that $\p^* E\subset E^{(1/2)}$.

Observe that $\Omega$, $O_1$ and $O_2$ are all sets of finite perimeter since they are open sets with smooth boundary (see e.g. \cite[Remark 5.4.2]{ziemer}) and their reduced boundaries coincide with the topological boundaries. The following lemma states that the set $\V$ defined in Section \ref{sec:2} is a set of finite perimeter.

  \begin{lemma}
    \label{prop:3}
    The set $\V$ defined by \eqref{def-Vset} is a set of finite perimeter.
  \end{lemma}
  \begin{proof}
  Since $O_1$ and $O_2$ have smooth boundary and $\p\V\subset \p O_1\cup \p O_2$, there holds
%The boundary of a connected component of $A \subset \mathbb{R}^n$ is a subset of $\partial A$.
  \begin{equation*}
    \H^{n-1}(\p\V)\leq \H^{n-1}(\p O_1\cup \p O_2)<\infty.
  \end{equation*}
  Due to \cite[Proposition 3.62]{ambrosio}, we conclude that $\V$ has finite perimeter.
  \end{proof}

\begin{mdframed}[backgroundcolor=black!4, linewidth=0pt, innerleftmargin=0pt, innerrightmargin=0pt]  
{\bf Notations.}  Let $E$ and $F$ be sets of finite perimeter in $\mathbb{R}^n$. We write
  \begin{align*}
      \left\{\nu_E=\nu_F \right\}&:=\left\{x\in \p^* E\cap\p^* F\mid \nu_E(x)=\nu_F(x) \right\},\\
      \left\{\nu_E=-\nu_F \right\}&:=\left\{x\in \p^* E\cap\p^* F\mid \nu_E(x)=-\nu_F(x) \right\}.
  \end{align*}
 To simplify notations, we write $S_1\simeq S_2$ if $S_1$ and $S_2$ are Borel sets such that $\H^{n-1}\left((S_1\setminus S_2)\cup(S_2\setminus S_1)\right)=0$, with respect to the $(n-1)$-dimensional Hausdorff measure $\H^{n-1}$. We write $S_1\subsetsim S_2$ if $S_1\subset S_2\cup N$ for some Borel set $N$ satisfying $\H^{n-1}(N)=0$.
\end{mdframed}

  \smallskip
  Lemma \ref{V-boundary} shows that $\p \V\setminus \p O_2\subset \p G_0$, it is natural to ask if the reduced boundary of these sets also preserves this relation. We remark that in general, for three Borel sets $A$, $B$ and $C$, $\p A\setminus \p B\subset \p C$ does not implies that $\p^* A\setminus \p^* B\subset \p^* C$ since $\H^{n-1}(\p C\setminus\p^* C)$ may not be $0$. 
  In our setting, the following lemma states that the reduced boundary satisfies this relation up to an $\H^{n-1}$-measure zero set.
  
  \begin{lemma}
  \label{prop:4}
      $\quad \p^* \V\subsetsim \p^* G_0\cup \p^* O_2$, where $G_0,\V$ are defined by $\eqref{def-G0},\eqref{def-Vset}$.
  \end{lemma}
  \begin{proof}
      By \cite[Theorem 16.3]{maggi}, we have
      \begin{equation*}
          \p^*\V=\p^*(\V\setminus G_0)\simeq \left(G_0^{(0)}\cap \p^* \V\right)\cup \left(\V^{(0)}\cap \p^* G_0\right)\cup \left\{ \nu_\V=-\nu_{G_0}\right\}.
      \end{equation*}
      It is enough to show that $G_0^{(0)}\cap \p^* \V \subset \p O_2=\p^* O_2$. To get a contradiction, suppose there exists a point $p\in G_0^{(0)}\cap \p^* \V$ and $p\not\in \p O_2$. Since $\p^*\V\subset \p\V\subset \p G_0\cup \p O_2=(\p G_0\cap \p O_1)\cup \p O_2$, there holds $p\in \p G_0\cap (\p O_1\setminus \p O_2)\cap G_0^{(0)}$. Thus, for any $\varepsilon>0$, there exists a small enough neighborhood $U_p^\varepsilon$ of $p$, such that $\frac{\abs{U_p^\varepsilon\cap G_0}}{\abs{U^\varepsilon_p}}\leq\varepsilon$. According to Lemma \ref{lm:new}, we have $U_p^\varepsilon\setminus G_0 \subset \overline{O_1}$. Hence $\frac{\abs{U_p^\varepsilon\cap O_1}}{\abs{U_p^\varepsilon}}=\frac{\abs{U_p^\varepsilon\cap \overline{O_1}}}{\abs{U_p^\varepsilon}}\geq 1-\varepsilon$. However, this contradicts with the fact that $p\in \p^* O_1\subset O_1^{(1/2)}$ due to Federer's theorem.
  \end{proof}
  
    Finally we show that the measure-theoretic unit normal on $\p^*\V$ coincides almost everywhere with the unit normal defined on $\p^* O_1$ or $\p^* O_2$, but with the opposite direction in the latter case. This is an application of the structure theory for sets of finite perimeter.

  \begin{proposition}
  \label{prop:5}
      $\quad \p^* \V\simeq \left\{\nu_{\V}=\nu_{O_1}\right\}\cup \left\{\nu_{\V}=-\nu_{O_2}\right\}.$
  \end{proposition}
  \begin{proof}
    According to Lemma \ref{lm:2} in Appendix \ref{apx}, we have $\p^*\V\cap \p^* G_0\simeq \left\{\nu_\V=-\nu_{G_0} \right\}$ and $\p^*\V\cap \p^* O_2\simeq\left\{\nu_\V=-\nu_{O_2} \right\}$ since $\V\cap G_0=\emptyset$ and $\V\cap O_2=\emptyset$. Moreover, $G_0\cap O_1=\emptyset$ and $G_0\cap O_2=\emptyset$ imply that $\p^*G_0\cap \p^* O_1\simeq\left\{\nu_{G_0}=-\nu_{O_1} \right\}$ and $\p^*G_0\cap \p^* O_2\simeq\left\{\nu_{G_0}=-\nu_{O_2} \right\}$. By Lemma \ref{prop:4}, we have $\p^*\V\subset \p^* O_1\cup \p^* O_2$ and $\p^* \V \subsetsim \p^* G_0\cup \p^* O_2$, then
    \begin{align*}
        \p^* \V &\simeq (\p^*\V \cap \p^*G_0\cap \p^* O_1)\cup (\p^*\V\cap \p^* O_2)\\
        &=\left((\p^* \V\cap \p^* G_0)\cap(\p^* G_0\cap \p^* O_1)\right)\cup (\p^*\V\cap \p^* O_2)\\
        &\simeq\left(\left\{\nu_\V=-\nu_{G_0} \right\}\cap \left\{\nu_{G_0}=-\nu_{O_1} \right\}\right)\cup\left\{\nu_\V=-\nu_{O_2} \right\}\\
        &= \left\{\nu_\V=\nu_{O_1} \right\}\cup\left\{\nu_\V=-\nu_{O_2} \right\}. \qedhere
    \end{align*}
  \end{proof}

\section{Inverse Obstacle Problem for the Laplace equation}
\label{sec:4}

In this section, we prove Theorem \ref{main-Laplace}.
Assume $O_1\neq O_2$, say, $O_1\not\subset O_2$ without loss of generality.
Let $G_0,\mathcal{V}$ be the sets defined by \eqref{def-G0},\eqref{def-Vset} at the beginning of Section \ref{sec:2}.
%Since $u_1,u_2$ and their normal derivatives coincide on $\partial \Omega$, by the unique continuation for elliptic equation (and path-connectedness), $u_1=u_2$ in the interior of $G_0$.

\begin{lemma}
\label{uc-G0}
$u_1=u_2$ in $\overline{G_0}$.
\end{lemma}
\begin{proof}
Since $u_1,u_2$ and their normal derivatives coincide on an open subset $\Gamma \subset \partial \Omega$, by the unique continuation for elliptic equation (and path-connectedness), $u_1=u_2$ in $G_0$.
They also coincide on $\partial G_0$ due to the regularity $u_1,u_2\in C^1(\overline{G_0})$.
\end{proof}

\begin{proof}[Proof of Theorem \ref{main-Laplace}]
We notice $u_2$ satisfies $\Delta u_2=0$ on $\mathcal{V}$ as $\mathcal{V}\cap \overline{O_2}=\emptyset$. 
On $\partial \mathcal{V} \cap \partial O_2$, the boundary condition applies: $u_2\partial_{\nu_2} u_2=0$, where $\nu_2$ is the unit normal vector of $\partial O_2$.
As $\partial \mathcal{V} \setminus (\partial \mathcal{V}\cap \partial O_2)\subset \partial O_1$ by Lemma \ref{V-boundary}, we know $u_1\partial_{\nu_1} u_1=0$ on $\p\V\setminus \p O_2$, where $\nu_1$ is the normal vector of $\partial O_1$. 
%Hence,
%\begin{equation} \label{condition-V}
%u_2\partial_{\nu} u_2=0 \textrm{ on }\partial \mathcal{V}.
%\end{equation}
%Now we want to use the integration by parts to show that $u_2=const$ in $V$, which would imply that the boundary value $f=const$, contradiction.
%reference Adams book.
Since $u_2\in C^1(\overline{\Omega}\setminus O_2)\cap H^2(\Omega\setminus \overline{O_2})$ is defined on $\Omega \setminus \overline{O_2}$ and extends to a neighborhood of $\Omega \setminus \overline{O_2}$ in the same regularity class (by the standard Sobolev extension), we can apply the Gauss-Green formula in \cite[Proposition 6.4]{monica2019} to $u_2$ in the set of finite perimeter $\mathcal{V}$,
%$u_2$ is defined outside $O_2$ and can be extended to interior of $O_2$ by Sobolev $H^2$-extension, preserving $C^1$. Sobolev extension needs $C^2$ boundary.
\begin{equation} \label{Gauss-Green}
\int_{\mathcal{V}} u_2 \Delta u_2 + \int_{\mathcal{V}} |\nabla u_2|^2= \int_{\partial^* \mathcal{V}} u_2 \nabla u_2 \cdot {\nu} \,d\mathcal{H}^{n-1}.
\end{equation}
Here $\partial^* \mathcal{V}$ is the reduced boundary of $\mathcal{V}$, where the measure-theoretical unit normal $\nu$ to $\mathcal{V}$ is defined. 
%well-defined in the sense of Federer. 
%Note that the reduced boundary used in Chen-Comi-Torres (2019) is the measure-theoretic boundary, which differs from Federer's definition by a measure zero set, see Def 5.5.1 in Ziemer's book.

%This can be argued by
%Theorem 16.3 in Francesco Maggi's book "Sets of Finite Perimeter and Geometric Variational Problems: an introduction to geometric measure theory".
%It states that on the intersection of the reduced boundary of two sets of finite perimeter, the normal coincides up to measure zero set.
%The proof is based on the structure theorem that the reduced boundary $\partial^* \mathcal{V}$ can be decomposed into a countable number of $C^1$ hypersurfaces and $\mathcal{H}^{n-1}$-measure zero set.
%see the definition of reduced boundary here, not the same as measure-theoretical boundary, but they differ by measure zero set.
%The definition of $E^{(0)}, E^{(1)}$ is in Example 5.17, the set of points of density $0,1$, i.e., measure-theoretical exterior and interior.
%By structure theorem (e.g. Def. 5.7.4), when $\mathcal{V}$ is set of finite perimeter, the reduced boundary $\partial^* \mathcal{V}$ (as defined in Def. 5.5.1 in Ziemer's book) can be decomposed into a countable number of $C^1$ hypersurfaces $M_i$ and a $\mathcal{H}^{n-1}$-measure zero set, i.e., $\partial^* \mathcal{V}= \cup_i M_i \cup N$ where $\mathcal{H}^{n-1}(N)=0$, see also Theorem 5.15 in Evans-Gariepy's book.

Using Proposition \ref{prop:5} and Signorini condition at $\partial O_2$, we have 
\begin{equation}
\label{eq:prof_2_1}
u_2 \nabla u_2  \cdot {\nu}=-u_2\partial_{\nu_2} u_2=0 \quad\textrm{ a.e. on }\,\partial^* \mathcal{V} \cap \partial O_2.
\end{equation}
If $\p^* \V\setminus \p O_2=\emptyset$, we have 
\begin{equation}
u_2 \nabla u_2  \cdot {\nu}=-u_2\partial_{\nu_2} u_2=0 \quad\textrm{ a.e. on }\,\partial^* \mathcal{V}.
\end{equation}
For the complement, suppose $\p^* \V\setminus \p O_2\neq\emptyset$,  
%\begin{equation}
%u_2 \nabla u_2 \cdot \nu= 0 \quad \textrm{ a.e. on }\,\partial^*\mathcal{V}\setminus (\partial^* \mathcal{V} \cap \partial O_2).
%\end{equation}
then take any point $z\in \partial^*\mathcal{V}\setminus\partial O_2$. According to Lemma \ref{V-boundary}, we have $z\in \p^*\V\setminus\p O_2\subset \p\V\setminus\p O_2\subset \p G_0\cap (\p O_1\setminus \p O_2)$. Thus, by Lemma \ref{lm:new}, for small enough neighborhood $U_z$ of $z$, there holds $U_z^-:=U_z\setminus \overline{O_1}\subset G_0$.

%we only need $C^1$ regularity.

In $U_z^-\subset G_0\subset \Omega\setminus (\overline{O_1\cup O_2})$, both $u_1$ and $u_2$ are defined and of $C^1$. Since $u_1=u_2$ in $U_z^-$ by Lemma \ref{uc-G0}, we have $u_1=u_2$ and $\partial_{\nu_1} u_1=\partial_{\nu_1}u_2$ at $z$. As the choice of $z$ is arbitrary, we have 
$$u_1=u_2,\text{ and }\partial_{\nu_1} u_1=\partial_{\nu_1}u_2 \quad \textrm{ on } \partial^*\mathcal{V}\setminus (\partial^* \mathcal{V} \cap \partial O_2).$$
Then by Proposition \ref{prop:5} and Signorini boundary condition for $u_1$ on $\partial O_1$, we have
\begin{equation}
\label{eq:prof_2_2}
u_2 \nabla u_2 \cdot \nu=u_2 \p_{\nu_1} u_2 =u_1 \p_{\nu_1} u_1= 0 \quad \textrm{ a.e. on }\,\partial^*\mathcal{V}\setminus (\partial^* \mathcal{V} \cap \partial O_2).
\end{equation}
Combining \eqref{eq:prof_2_1} and \eqref{eq:prof_2_2}, and the case $\p^*\V\setminus\p O_2=\emptyset$, we have shown that 
\begin{equation}
u_2 \nabla u_2 \cdot \nu=0 \quad \textrm{ a.e. on }\,\partial^*\mathcal{V}.
\end{equation}
Thus from \eqref{Gauss-Green} and $\Delta u_2=0$ in $\V$, it follows that $\int_{\mathcal{V}}|\nabla u_2|^2=0$, and consequently $u_2$ is a constant in $\mathcal{V}$. This implies that the boundary value $f$ is a constant on $\partial \Omega$ by Lemma \ref{V-boundary} and unique continuation, which contradicts that $f$ is not constant on $\partial \Omega$ and thus conclude the proof of Theorem \ref{main-Laplace}.
\end{proof}

\section{Inverse Obstacle Problem for Elasticity System}
\label{sec:5}
  In this section, we extend the result we obtained in Section \ref{sec:4} to the L\'ame system. Let $\bu_1$ and $\bu_2$ be the solution of \eqref{eq:direct_problem_ela} with obstacles $O_1$ and $O_2$ respectively. Provided $\bs(\bu_1) \n|_\Gamma=\bs(\bu_2) \n|_\Gamma$ for an open subset $\Gamma$ of $\p\Omega$, we see that $\bu_1$ and $\bu_2$ coincide in $\overline{G_0}$, analogous to Lemma \ref{uc-G0}.
  \begin{lemma}
      $\bu_1=\bu_2$ in $\overline{G_0}$.
  \end{lemma}
  \label{lm:uc_lame}
  \begin{proof}
      Since $\bu_1|_\Gamma=\bu_2|_\Gamma$ and $\bs(\bu_1) \n |_\Gamma=\bs(\bu_2) \n |_\Gamma$, by the unique continuation for local Cauchy data \cite[Corollary 2.2]{eberle2021} and the path-connectedness of $G_0$, there holds $\bu_1=\bu_2$ in $G_0$. Moreover, $\bu_1$ and $\bu_2$ also coincide on $\p G_0$ due to the regularity $\bu_1,\bu_2\in C^1(\overline{G_0})$.
  \end{proof}\par
  Our next objective is to prepare Green's formula for the stress tensor on $\V$. To this end, we generalize Green's formula \cite[eq. (4.22)]{sofonea} from Lipschitz domains to sets of finite perimeter based on \cite{monica2019}.
  \begin{lemma}
    \label{lm:gauss}
    Let $E$ be an open set in $\R^n$ and $\bu\in \left(C^1(E)\right)^n\cap \left(H^2(E)\right)^n$. Let $U\subset\subset E$ be an open set of finite perimeter. Then
    %For $\bs(\bu)= 2\mu \be(\bu)+\lambda(\tr\be(\bu))I_n$, there holds
    \begin{align}
      \label{eq:Gauss}
      \int_U \bs(\bu): \be(\bu)\ dx+\int_U \bu \cdot \div \bs(\bu)\ dx=\int_{\p^* U}\bs(\bu)\bn\cdot \bu \ d\mathcal{H}^{n-1},
    \end{align}
    where $\bn$ is the measure-theoretic unit normal on the reduced boundary $\p^* U$ of $U$.
  \end{lemma}
  
  \begin{proof}
    Since $\be(\bu)=\frac{1}{2}\left(\nabla \bu+(\nabla \bu)^T\right)$, then to prove \eqref{eq:Gauss}, it is sufficient to show that
    \begin{equation}
      \begin{cases}
        \label{case}
        \int_U \bs(\bu): \nabla \bu\ dx+\int_U \bu\cdot \div \bs(\bu)\ dx=\int_{\p^* U}\bs(\bu)\bn\cdot \bu \ d\mathcal{H}^{n-1} \\
        \\
        \int_U \bs(\bu): (\nabla \bu)^T\ dx+\int_U \bu\cdot \div \bs(\bu)\ dx=\int_{\p^* U}\bs(\bu)\bn\cdot \bu \ d\mathcal{H}^{n-1}
      \end{cases}
    \end{equation}
    As $\bs(\bu)$ is symmetric, the above two equations are exactly the same. Write $\bs_i=(\sigma_{i1},\sigma_{i2},\cdots,\sigma_{in})$ the $i$-th row of the matrix $\bs(\bu)$, then proving \eqref{case} amounts to showing that 
    \begin{equation}
      \label{eq:each_gauss}
      \int_U \bs_i\cdot \nabla u_i\ dx+\int_U u_i\ \div \bs_i \ dx=\int_U \div (u_i\bs_i)\ dx=\int_{\p^* U}(\bs_i\cdot\bn) u_i \ d\mathcal{H}^{n-1},\ 1\leq i\leq n.
    \end{equation} 
    Recall that
    \begin{align}
    \label{eq:express_bs}
      \sigma_{ij}=\mu\left(\frac{\p u_i}{\p x_j}+\frac{\p u_j}{\p x_i}\right)+\delta_i^j \lambda \left(\sum_{k=1}^{n}\frac{\p u_k}{\p x_k}\right),
    \end{align}
    where $\mu,\lambda\in C^\infty(E)$. Since $\bu\in (H^2(E))^n$, \eqref{eq:express_bs} implies that $\div\bs_i\in L^2(E)$, in particular, one can view $\abs{\div \bs_i}$ as a real Radon measure on $\Omega$ and thus $\bs_i$ is a divergence measure field \cite[Definition 2.3]{monica2019}. Notice that $\bs_i\in C^0(E;\R^n)$ and $u_i\in C^1(E)$, then by \cite[Proposition 6.3]{monica2019}, there holds \eqref{eq:each_gauss} and thus we obtain \eqref{eq:Gauss}.
  \end{proof}
  \begin{proof}[Proof of Theorem \ref{main:elasticity}]
    Without loss of generality, we assume $O_1\not\subset O_2$. We denote the measure-theoretical unit normal of $\V$, $O_1$ and $O_2$ by $\bn$, $\bn_1$, and $\bn_2$, respectively.\par
    Using the Signorini boundary condition on $\p O_2$ and Proposition \ref{prop:5}, we have
    \begin{equation}
    \label{eq:thm_ela_1}
        \bs(\bu_2)_\tau=0,\ (\bu_2)_{\nu}\bs(\bu_2)_{\nu}=(\bu_2)_{\nu_2}\bs(\bu_2)_{\nu_2}=0 \quad\text{a.e. on }\ \p^*\V\cap \p O_2.
    \end{equation}
    
    Analogous to the argument used in the proof for Theorem \ref{main-Laplace} in Section \ref{sec:4}, we can obtain
    \begin{equation}
        \bs(\bu_2)_\tau=\bs(\bu_1)_\tau=0,\ \bs(\bu_1)_{\nu_1}=\bs(\bu_2)_{\nu_1}\quad\text{on  }\p^*\V\setminus \p O_2.
    \end{equation}
    By Lemma \ref{V-boundary}, $\p^* \V\setminus \p O_2\subset \p G_0$, and then Lemma \ref{lm:uc_lame} reads that $\bu_1=\bu_2$ on $\p^* \V\setminus \p O_2$. According to Proposition \ref{prop:5} and the Signorini boundary condition holds for $\bu_1$ on $\p O_1$, there holds
    \begin{align}
        \label{eq:thm_ela_2}
        \bs(\bu_2)_{\tau}=\bs(\bu_2)_{\tau_1}=&\bs(\bu_1)_{\tau_1}=0,\\ 
        (\bu_2)_\nu\bs(\bu_2)_\nu=(\bu_2)_{\nu_1}\bs(\bu_2)_{\nu_1}=&(\bu_1)_{\nu_1}\bs(\bu_1)_{\nu_1}=0\quad \text{a.e. on }\p^*\V\setminus\p O_2.
    \end{align}
    Combining \eqref{eq:thm_ela_1} and \eqref{eq:thm_ela_2}, we have
    \begin{equation}
         \bs(\bu_2)_{\tau}=0,\ \bs(\bu_2)\bn\cdot\bu_2=(\bu_2)_\nu\bs(\bu_2)_\nu=0\quad\text{on }\p^*\V.
    \end{equation}
    Thanks to the regularity results \cite[Theorem 2.2]{kinderlehrerF} and \cite[Theorem 3.10]{schumann1989}, there holds $\bu_2\in (C^1(\overline{\Omega}\setminus O_2))^n\cap (H^2(\Omega\setminus\overline{O_2}))^n$. By the standard Sobolev extension, we can extend $\bu_2$ to a neighborhood $U$ of $\Omega\setminus O_2$ such that $\V\subset\subset U$ and $\bu_2\in (C^1(U))^n\cap (H^2(U))^n$. Then applying Lemma \ref{lm:gauss}, there holds
     \begin{align}
      \int_\V \bs(\bu_2): \be(\bu_2)\ dx&=\int_{\p^* \V}\bs(\bu_2)\bn\cdot \bu_2\ d\mathcal{H}^{n-1}\\
      &=\int_{\p^*\V} (\bs(\bu_2)_\tau+\bs(\bu_2)_\nu \bn)\cdot \bu_2\ d\mathcal{H}^{n-1}\\
      &=\int_{\p^* \V}\bs(\bu_2)_\nu(\bu_2)_\nu\ d\mathcal{H}^{n-1}=0.
    \end{align}
    Since 
    \begin{equation}
      \bs(\bu_2): \be(\bu_2)=2\mu \be(\bu_2): \be(\bu_2)+\lambda (\tr \be(\bu_2))^2,
    \end{equation}
    where $\mu,\lambda>0$ in $\overline{\Omega}$. We can conclude that $\be(\bu_2)= 0$ in $\V$. \par
    According to Lemma \ref{lm:skew} in Appendix \ref{apx}, we have $\bu_2=\boldsymbol{c}+A\x$ in $\V$, where $\boldsymbol{c}\in \R^n$ and $A\in \R^{n\times n}$ is constant a skew-symmetric matrix. Then due to the unique continuation for the static elasticity system (see e.g. \cite{ang1998} or \cite[Theorem 2.3]{uhlmann2009}), $\bu_2= \boldsymbol{c}+A\x$ in $\Omega\setminus \overline{O_2}$ and $\boldsymbol{f}=(\boldsymbol{c}+A\x)|_{\p\Omega}$, which contradicts with the assumption $\boldsymbol{f}\not\in \mathcal{R}$.
  \end{proof}

  \section{Counterexamples to unique solvability}
\label{sec-counterexamples}

In this section, we discuss counterexamples to the unique solvability of the inverse Signorini obstacle problem.
For the scalar problem, it is not possible to solve the problem if the Dirichlet boundary value $f$ is a nonnegative constant.
Indeed, in this case the constant function $u=c\geq 0$ is the solution to the Signorini problem by the uniqueness of solution, since $u=c\geq 0$ is allowed by the Signorini condition $u|_{\p O} \geq 0$.
This means that given any nonnegative constant boundary value,
the normal derivative of the solution for any Signorini obstacle is identically zero.

As a special case in the setting of Theorem \ref{main-Laplace}, we show that the inverse obstacle problem is uniquely solvable if $f$ is a negative constant.

      \begin{proposition}
      \label{prop:nec_lap}
        Let $\Omega\subset \mathbb{R}^n$ be a bounded open set with smooth boundary and $\Gamma\subset \partial \Omega$ be a nonempty open subset. Let $O_1,O_2\subset\subset \Omega$ be open subsets with smooth boundary.
        Assume that $\Omega\setminus \overline{O_1},\Omega\setminus \overline{O_2}$ are connected. 
        Suppose $u_1,u_2$ solves \eqref{eq:direct_problem_lap} for the given boundary value $f$ with obstacle $O_1,O_2$, respectively. 
        If $f$ is a negative constant and $\partial_{n} u_1|_{\Gamma}=\partial_{n} u_2|_{\Gamma}$ with respect to the unit normal $\boldsymbol{n}$ to $\partial \Omega$, then $O_1=O_2$.
      \end{proposition}
      
      \begin{proof}
          Assume $O_1\neq O_2$, say $O_1\not\subset O_2$ without loss of generality.
          Following the proof of Theorem \ref{main-Laplace}, we have $u_2=c$ for some constant $c$ in $\overline{\Omega}\setminus O_2$. 
          Under the assumption that $f$ is a negative constant, this implies that $u_2=c<0$.
          However, this contradicts the Signorini condition $u_2|_{\p O_2}\geq 0$, as long as $O_2\neq \emptyset$.
          In the case of $O_2=\emptyset$, Lemma \ref{uc-G0} shows $u_1=u_2=c<0$ on $\overline{\Omega} \setminus O_1$, which again contradicts with the Signorini condition $u_1|_{\p O_1}\geq 0$ as long as $O_1\neq \emptyset$. Hence both $O_1,O_2$ have to be empty set.
          \end{proof}

For the elasticity case, the solvability issue is more complicated because the shape of the obstacle matters to the Signorini condition $\bu \cdot \nu\leq 0$.
              Consider the sets 
      \begin{equation}
          \Xi := \big\{O\subset \Omega \mid O\text{ is a nonempty open subset with smooth boundary} \big\},
      \end{equation}
      and we define
      \begin{equation}          \Upsilon_{A,\boldsymbol{c}}:=\big\{O\subset \Xi\mid (\boldsymbol{c}+A\x)\cdot\nu_O(\x)=0\text{ for all }\x\in \p O \big\} \cup \{\emptyset\},
      \end{equation}
      where $\nu_O$ denotes the inward normal of $O$ at $\p O$.
Complementing Theorem \ref{main:elasticity}, in the following lemma we characterize the unique solvability in the case of $\boldsymbol{f}\in \mathcal{R}$.
      \begin{proposition}
      \label{prop:nec_ela}
    Let $\Omega\subset \mathbb{R}^n$ be a bounded open set with smooth boundary and $\Gamma\subset \partial \Omega$ be a nonempty open subset. Let $O_1,O_2\subset\subset \Omega$ be open subsets with smooth boundary.
    Assume that $\Omega\setminus \overline{O_1},\Omega\setminus \overline{O_2}$ are connected.
    Suppose $\bu_1, \bu_2$ solves the elasticity system \eqref{eq:direct_problem_ela} for $\boldsymbol{f}=\boldsymbol{c}+A\x \in\mathcal{R}$ with $O_1, O_2$, respectively, for some vector $\boldsymbol{c}\in \mathbb{R}^n$ and skew-symmetric matrix $A$.
If  $\bs(\bu_1)\n|_{\Gamma}=\bs(\bu_2)\n |_{\Gamma}$, then $O_1=O_2$ unless $O_1,O_2\in  \Upsilon_{A,\boldsymbol{c}}$.
      \end{proposition}
      \begin{proof}
          It suffices to show that $O_1\neq O_2$ implies $O_1,O_2\in  \Upsilon_{A,\boldsymbol{c}}$. Assume $O_1\not\subset O_2$ without loss of generality. Following the proof of Theorem \ref{main:elasticity} and the continuity of solutions, we have $\bu_2=\boldsymbol{c}+A\x$ in $\overline{\Omega}\setminus O_2$.

Consider the case when $O_2\neq \emptyset$. The function $\bu_2$ clearly extends smoothly to $\widetilde{\bu}_2:=\boldsymbol{c}+A\x$ in the whole domain $\overline{\Omega}$. Then we apply the divergence theorem to $\widetilde{\bu}_2$ in $O_2$,          
          \begin{equation}
              \int_{\p O_2}(\boldsymbol{c}+A\x)\cdot \nu_{O_2}(\x) =-\int_{O_2}\div (\boldsymbol{c}+A\x)=0.
          \end{equation}
          %Here uses the inward direction of unit normal of $O$, to match the unit normal outward direction of $\Omega\setminus O$.
          Since the Signorini condition for $\bu_2$ on $\p O_2$ requires 
          \begin{equation}
              (\bu_2)_\nu=(\boldsymbol{c}+A\x)\cdot \nu_{O_2}(\x)\leq 0, \quad\textrm{for }\x\in \p O_2,
          \end{equation}
          we have $(\boldsymbol{c}+A\x)\cdot\nu_{O_2}(\x)=0$ on $\p O_2$ and thus $O_2\in \Upsilon_{A,\boldsymbol{c}}$. 
          As $\bu_2=\boldsymbol{c}+A\x$ in $\overline{\Omega}\setminus O_2$, it follows that $\bs(\bu_2)=0$ identically and thus $\bs(\bu_1)\n|_{\Gamma}=\bs(\bu_2)\n |_{\Gamma}=0$. Then the unique continuation property yields that $\bu_1=\boldsymbol{c}+A\x$ in $\overline{\Omega}\setminus O_1$. Then the same argument as above gives $O_1\in \Upsilon_{A,\boldsymbol{c}}$ (since $O_1\neq \emptyset$ as $O_1\not \subset O_2$).
          In the case of $O_2=\emptyset$, one repeats the argument for $\bu_1$ to show $O_1\in \Upsilon_{A,\boldsymbol{c}}$.
      \end{proof}

A corollary of Proposition \ref{prop:nec_ela} is that if $\Upsilon_{A,\boldsymbol{c}}= \{\emptyset\}$, then the inverse obstacle problem is uniquely solvable with boundary data $\boldsymbol{f}=\boldsymbol{c}+A\x$.
The vector $\boldsymbol{c}$ and matrix $A$ determine if $\Upsilon_{A,\boldsymbol{c}}$ is empty or not.
Next, we give examples for both cases $\Upsilon_{A,\boldsymbol{c}}\neq \{\emptyset\}$ and $\Upsilon_{A,\boldsymbol{c}}= \{\emptyset\}$.
      \begin{example}
          For any given constant skew-symmetric matrix $A$ and a point $p\in \Omega$, we set $\boldsymbol{c}=-Ap$. Then there exists a ball $B$ centered at $p$ with radius $\delta>0$ such that $B\subset \Omega$. Notice that for any $q\in \p B(p,\delta)$, the normal derivative $\nu_B$ at $q$ is $(q-p)/\delta$. Since $A$ is skew-symmetric, we have
          \begin{equation}
              (\boldsymbol{c}+Aq)\cdot \nu_B(q)=\frac{1}{\delta}A(q-p)\cdot (q-p)=0.
          \end{equation}
          Therefore, $B(p,\delta)\in \Upsilon_{A,\boldsymbol{c}}$.
      \end{example}
      \begin{example}
          For any given constant skew-symmetric matrix $A$, the range of $A$ restricted in $\Omega$ is bounded, i.e. for a large enough $M$, we have
          \begin{equation}
              \max_{\x\in \Omega}\, \abs{A\x}<M.
          \end{equation}
          Then we choose a vector $\boldsymbol{c}$ with $\abs{\boldsymbol{c}}>M$. Assume that there exists a nonempty set $O\in \Upsilon_{A,\boldsymbol{c}}$.
         To get a contradiction, we choose a point $z\in \p O$ such that $\nu_O(z)=\boldsymbol{c}/\abs{\boldsymbol{c}}$. 
         Such a point $z$ can be chosen in the following way. Let $P$ be a hyperplane perpendicular to $\boldsymbol{c}$. Initially place the hyperplane $P$ disjoint from the set $O$, and move the hyperplane towards $P$ until they first intersect. At any intersection points, it must satisfy that $P$ is tangential to $\partial O$ so that the normal directions of $P$ and $\partial O$ coincide.
         %here uses the smoothness of the boundary of $O$.
         Then at point $z$, we have
          \begin{equation}
            \abs{(\boldsymbol{c}+A z)\cdot \nu_O(z)}\geq\abs{\boldsymbol{c}}-\abs{Az\cdot\frac{\boldsymbol{c}}{\abs{\boldsymbol{c}}}}> \abs{\boldsymbol{c}}-M>0.
          \end{equation}
          However, this contradicts with the assumption that $(\boldsymbol{c}+A \x)\cdot \nu_O(\x)=0$ for all $\x\in \p O$. Thus $\Upsilon_{A,\boldsymbol{c}}= \{\emptyset\}$.        
      \end{example}

  \appendix
  \section{Auxiliary lemmas}
  \label{apx}
 \begin{lemma}
    \label{lm:1}
    Let $A, B\subset \R^n$ be two nonempty connected open sets satisfying $A\subset B$. If $\p A\subset \p B$, then $A=B$.
  \end{lemma}
  \begin{proof}
    Assume $A\neq B$, then there exists a point $p\in B$ but $p\not\in A$. If $p\in \overline{A}$, then $p\in (\overline{A}\b A)=\p A\subset \p B$ and $p\not\in B$, a contradiction. On the contrary, if $p\not\in \overline{A}$, then $p\in B\b \overline{A}\neq \emptyset$, and we have
    \begin{align*}
      (B\b \overline{A})\cup A=B\b \p A\supset B\b \p B=B.
    \end{align*}
    Notice that $(B\b \overline{A})\cap A=\emptyset$. Thus, the set $B$ can be written as the union of two disjoint open sets $B\b \overline{A}$ and $A$. 
    As $A\neq \emptyset$, it must satisfy that $A=B$ due to the connectedness of $B$.
  \end{proof}

    \begin{lemma}
  \label{lm:2}
      If $E$ and $F$ are sets of finite perimeter, then there holds
      \begin{equation*}
          \p^* E\cap \p^* F\simeq \left\{\nu_E=\nu_F \right\}\cup \left\{\nu_E=-\nu_F \right\}.
      \end{equation*}
      Furthermore, if $E\cap F=\emptyset$, then
      \begin{equation*}
          \p^* E\cap \p^* F\simeq \left\{\nu_E=-\nu_F \right\}.
      \end{equation*}
  \end{lemma}
  \begin{proof}
      According to \cite[Theorem 16.3]{maggi}, we have
      \begin{align*}
          \p^*(E\cap F)&\simeq \left(F^{(1)}\cap \p^* E\right)\cup \left(E^{(1)}\cap \p^* F\right)\cup \left\{ \nu_E=\nu_F\right\},\\
          \p^*(E\setminus F)&\simeq \left(F^{(0)}\cap \p^* E\right)\cup \left(E^{(1)}\cap \p^* F\right)\cup \left\{\nu_E=-\nu_F\right\},\\
          \p^*(E\cup F)&\simeq \left(F^{(0)}\cap \p^* E\right)\cup \left(E^{(0)}\cap \p^* F\right)\cup \left\{ \nu_E=\nu_F\right\}.
      \end{align*}
      Since $E\cap F$, $E\cup F$ and $E\setminus F$ are all sets of finite perimeter (see e.g. \cite[Lemma 12.22]{maggi}), we can obtain
      \begin{align*}
          \p^* E&=\p^*\left((E\setminus F)\cup (E\cap F)\right)\\
          &\simeq \left((E\setminus F)^{(0)}\cap \p^*(E\cap F)\right)\cup \left((E\cap F)^{(0)}\cap \p^*(E\setminus F)\right)\cup \left\{\nu_{E\setminus F}=\nu_{E\cap F}\right\}\\
          &\subset \p^*(E\cap F)\cup \p^* (E\setminus F)
      \end{align*}
      as $\left\{\nu_{E\setminus F}=\nu_{E\cap F}\right\}\subset \p^*(E\cap F)\cap \p^* (E\setminus F)$ by definition. \par
      By Federer's theorem, $\p^* E\cap \p^* F\subset E^{(1/2)}\cap F^{(1/2)}$. Therefore we have
      \begin{align*}
          \p^* E\cap \p^* F\subsetsim \left\{ \nu_E=\nu_F\right\}\cup \left\{ \nu_E=-\nu_F\right\}.
      \end{align*}
      Since $\left\{ \nu_E=\nu_F\right\}\cup \left\{ \nu_E=-\nu_F\right\}\subset \p^* E\cap \p^* F$ by definition, then there follows $\p^* E\cap \p^* F\simeq \left\{\nu_E=\nu_F \right\}\cup \left\{\nu_E=-\nu_F \right\}$.
 
      Similarly, when $E\cap F=\emptyset$, i.e., $E=E\setminus F$, we have $\p^* E=\p^* (E\setminus F)$ and $\p^* E\cap \p^* F\simeq \left\{\nu_E=-\nu_F \right\}$ follows.
  \end{proof}
  \begin{lemma}
  \label{lm:skew}
          For any open subset $U\subset\subset \Omega$, if $\bu\in (H^2(U))^n\cap (C^1(U))^n$ satisfies $\be(\bu)=0$, then $\bu=\boldsymbol{c}+A\x$ in $U$, where $\boldsymbol{c}\in \R^n$ is a constant vector and $A\in \R^{n\times n}$ is a constant skew-symmetric matrix.
      \end{lemma}
      \begin{proof}
          Write $\bu=(u_1,u_2,\cdots,u_n)$, then for $1\leq i,j\leq n$, $\be(\bu)=\nabla \bu+(\nabla \bu)^T=0$ reads that
          \begin{align}
              \frac{\p u_i}{\p x_i}=0,\ \frac{\p u_j}{\p x_i}=-\frac{\p u_i}{\p x_j}.
          \end{align}
          Since $u_k\in H^2(U)$ for $1\leq k\leq n$, there holds
            \begin{equation}
                \frac{\p^2 u_k}{\p x_i\p x_j}=\frac{\p^2 u_k}{\p x_j\p x_i} \;\text{ in } L^2(U).
            \end{equation}
            Therefore, for any $1\leq i,j,k\leq n$, we have
            \begin{equation}
                 \frac{\p^2 u_k}{\p x_i\p x_j}=- \frac{\p^2 u_i}{\p x_k\p x_j}= \frac{\p^2 u_j}{\p x_i\p x_k}=- \frac{\p^2 u_k}{\p x_i\p x_j}=0.
            \end{equation}
            This shows that $\frac{\p u_i}{\p x_j}$ is constant for all $i,j$, and thus $\nabla \bu\in \R^{n\times n}$ is a constant skew-symmetric matrix. Due to $\bu\in (C^1(U))^n$, we obtain $\bu=\boldsymbol{c}+A\x$ in $U$ with a constant vector $\boldsymbol{c}\in \R^n$ and a constant skew-symmetric matrix $A=\nabla \bu\in \R^{n\times n}$.
      \end{proof}

\bigskip
\bibliography{reference}
\bibliographystyle{plain}

\end{document}